\documentclass[12pt]{amsart}
\usepackage{amssymb}
\usepackage{verbatim}
\usepackage{amscd}

\theoremstyle{plain}
\newtheorem{theorem}{Theorem}
\newtheorem{corollary}{Corollary}

\newtheorem{lemma}{Lemma}

\theoremstyle{definition}
\newtheorem{definition}{Definition}

\theoremstyle{remark}
\newtheorem{remark}{Remark}
\newtheorem{example}{Example}
\newtheorem*{notation}{Notation}

\def\bbN{\mathbb N}

\def\bbR{\mathbb R}

\def\bbZ{\mathbb Z}

\newcommand{\fB}{\mathfrak B}     \newcommand{\sB}{\mathcal B}

     \newcommand{\sN}{\mathcal N}
     \newcommand{\sO}{\mathcal O}
     \newcommand{\sP}{\mathcal P}

     \newcommand{\sV}{\mathcal V}

\newcommand{\al}{\alpha}

\newcommand{\ep}{\epsilon}
\newcommand{\vpi}{\varphi}
\newcommand{\si}{\sigma}
\newcommand{\la}{\lambda}
\newcommand{\de}{\delta}

\begin{document}

% topmatter
\title{Continuity of Translation Operators}

\author{Krishna B. Athreya and
Justin R. Peters}
\address{Department of Mathematics\\
    Iowa State University, Ames, Iowa}
\email{kba@iastate.edu}
\email{peters@iastate.edu}
\thanks{The second author acknowledges partial support from the National Science Foundation, DMS-0750986}


% \keywords{keywords}
% \thanks{2000 {\itshape Mathematics Subject Classification}.
%  Primary, *****; Secondary, *****.}
% \date{\date}

\begin{abstract} For a Radon measure $\mu$ on $\bbR,$ we show that
$L^{\infty}(\mu)$ is invariant under the group of translation
operators $T_t(f)(x) = \mbox{$f(x-t)$}\ (t \in \bbR)$ if and only if
$\mu$ is equivalent to Lebesgue measure $m$. We also give necessary
and sufficient conditions for $L^p(\mu),\ 1 \leq p < \infty,$ to be
invariant under the group $\{ T_t\}$ in terms of the Radon-Nikodym
derivative w.r.t. $m$.
\end{abstract}

\maketitle
\section{Introduction}
One of the celebrated theorems of analysis is the existence of a
unique measure $m$ on $\bbR$ such that $m([0, 1]) = 1$ and $m(E + t)
= m(E)$ for any Borel subset $E \subset \bbR,$ for all $t \in \bbR.$
This theorem was later extended by Haar (\cite{Ha}) to the setting
of locally compact groups.

We could rephrase the theorem as follows: there is a unique Radon
measure $m$ on $\bbR$ for which $m([0, 1]) = 1$ and translation acts
as an isometry on $L^p(m),$ for some $p,\ 1 \leq p < \infty.$ That
is, for all $f \in L^p(m),\ ||T_t(f)||_p = ||f||_p, $ for all $t \in
\bbR,$ where $T_t(f)(x) = f(x-t).$

In this note, we ask the question: which Radon measures $\mu$ have
the property that $T_t$ maps $L^p(\mu)$ into itself, for all $t \in
\bbR$?  We make no assumption of continuity of the $T_t.$

We show that for $1 \leq p < \infty$ if the translation operators
$T_t$ map $L^p(\mu)$ into itself, for all $t \in \bbR,$ then each
$T_t$ is strongly continuous, and moreover for any compact
neighborhood $\sN$ of $0,$ there is a constant $C$ so that $ ||T_t||
\leq C$ for all $t \in \sN$ (Corollary~\ref{c:boundedtrans}) and
thus the group $T_t$ acts as a $C_0-$group on $L^p(\mu)$, i.e.
$\lim_{t \to 0} ||T_t(f) - f||_p = 0.$ (Theorem~\ref{t:transcont}).
On the other hand, if $T_t$ acts on $L^{\infty}(\mu)$ for all $t \in
\bbR$, the action is not even weakly continuous
(Theorem~\ref{t:notwklycont}). In Theorems~\ref{t:linfty} and
\ref{t:transmeas} we give necessary and sufficient conditions for
$T_t$ to map $L^p(\mu)$ to itself in the cases $p = \infty$ and $p <
\infty$ respectively.

A related problem was treated by D. Bell (\cite{Be}). He called a
Radon measure $\mu$ \emph{quasi-translation invariant} if the null
sets of $\mu$ are translation invariant. He showed that with
additional assumptions on $\mu$, $\mu$ must be equivalent to
Lebesgue measure. Bell's results are valid in the context of locally
compact abelian groups. We deduce Bell's results from our results
for the group $\bbR.$ Lamperti \cite{Lam} considered a general class
of isometries of $L^p(X, \mu),$ for $ p < \infty.$  Our results do
not overlap with his since the mappings we consider are not
isometries, except for the case of $L^{\infty}(\mu).$

While it is not difficult to see that some of our results may be
recast in a more general setting, beginning with
Lemma~\ref{l:intervals} we are restricted to the setting of $\bbR.$
For thematic unity, we have presented all results in the context of
$\bbR.$

\subsection{Background and Notation}  A
Radon measure is positive Borel measure $\mu$ on $\bbR$  such that
$\mu([-n, n]) < \infty$ for all $n \in \bbN$. The vector space
$C_{00}(\bbR)$ of continuous functions with compact support, has as
its dual space the vector space of signed Radon measures. The space
$C_{00}(\bbR)$ is dense in $L^p(\mu), \ 1 \leq p < \infty.$  The
space $L^{\infty}(\mu)$ is the dual space of $L^1(\mu).$

If $f$ is an extended real-valued function defined on $\bbR,\
T_t(f)$ will denote the function $T_t(f)(x) = f(x-t), \ t \in \bbR.$

Certain hypotheses will appear often, and so it is convenient to
state them here. \\
\textbf{Hypotheses} \\
\textbf{H1} $\mu$ is a Radon measure with the property that if $f,\
g $ are extended real-valued Borel functions such that $ f = g \
\text{a.e.} [\mu]$, then for all $t \in \bbR,\ T_t(f) = T_t(g) \ \text{a.e.} [\mu].$\\
\textbf{H2(a)} Given a Radon measure $\mu$ and for some $p,\ 1 \leq
p < \infty,$ if $f$ is a Borel function such that
\[ \int_{\bbR} |f(x)|^p \,d\mu(x) < \infty, \text{ then }
\int_{\bbR} |T_t(f)(x)|^p\, d\mu(x) < \infty, \text{ for all } t \in
\bbR.\] \\
\textbf{H2(b)} Given a Radon measure $\mu,$ an extended real-valued
Borel function $f$ is essentially $[\mu]-$bounded if and only if
$T_t(f)$ is essentially $[\mu]-$bounded, for all $t \in \bbR.$

In \cite{Be} Bell made the following
\begin{definition} \label{d:qti} A Radon measure $\mu$ is
called \emph{quasi-translation invariant (q.t.i.)} if
\[ \forall E \in \fB,\ \mu(E) = 0 \iff \mu(E+t) = 0 \text{ for all } t \in \bbR\]
where $\fB$ is the $\si$-algebra of Borel sets.
\end{definition}

While the following result is obvious, it is nonetheless worth
recording.
\begin{lemma} \label{l:equivclass}
For a Radon measure $\mu$, hypothesis (H1) is equivalent to q.t.i.
\end{lemma}

\begin{remark} \label{r:equivalentcond} It is easy to see that each
of hypotheses H2(a),\ H2(b) implies H1, but we find it useful to
state the conditions separately.
\end{remark}

\begin{remark} \label{r:group}
(H1) implies that the group of translation operators $\{T_t\}_{t \in
\bbR}$ not only acts on functions, but acts as a group of linear
operators on equivalence classes of functions. By that we mean that
each $T_t$ is a linear operator, and $T_s\circ T_t = T_{s+t}$ for
all $s,\ t \in \bbR,$ and $T_0 = I. $ (H1) and (H2(a)) (resp., (H1)
and H2(b))) together imply that the group $\{T_t\}_{t\in \bbR}$ acts
on $L^p(\mu) $ (resp., $L^{\infty}(\mu).$)

Conversely, if $\{T_t\}_{t \in \bbR}$ acts as a group of linear
operators on some $L^p(\mu)$, then $\mu$ must satisfy (H1), or
equivalently q.t.i. Because if $\mu(E) > 0$ but $\mu(E+t) = 0$ for
some $t \in \bbR,$ then $\chi_E = T_0(\chi_E) = T_{-t}\circ
T_t(\chi_E) = T_{-t}(0) = 0.$  This contradiction shows that the
group property of the linear operators $\{T_t\}$ on $L^p(\mu)$ is
equivalent to q.t.i.
\end{remark}

From now on $m$ will denote Lebesgue measure.

\section{Continuity of Translation}

\begin{lemma}\label{l:autocont}
Let $\mu$ be a Radon measure  satisfying either (H1) and (H2(a)), or
(H1) and (H2(b)). Then for each $t \in \bbR,\ T_t$ is continuous.
\end{lemma}

\begin{proof}
Let $\sV = L^p(\mu)$. By Remark~\ref{r:group}, the maps $\{T_t\}$
act as a group on $\sV.$ Thus $T_t(\sV) = \sV$ so that  $T_t$ is
onto, and in particular has closed range.  It now follows from the
Closed Graph Theorem (\cite{Ru1}) that for each $t \in \bbR,\ T_t$
is strongly continuous on $L^p(\mu)$.
\end{proof}

\begin{remark} \label{r:continuity} The simple application above of the
Closed Graph Theorem yields that the maps $T_t$ are individually
continuous. A more difficult question is whether the \emph{group}
$\{T_t\}_{t\in \bbR}$ is continuous in some sense. We will see there
is a sharp distinction between the cases $p < \infty$ and $p =
\infty.$
\end{remark}

\begin{definition} \label{d:contmeas} A Radon measure $\mu$ is
\emph{continuous} if for all $x \in \bbR,$\ \mbox{$\mu(\{x\}) = 0.$}
\end{definition}

\begin{lemma} \label{l:support}
Let $\mu$ satisfy (H1) and (H2(a)), or (H1) and (H2(b)). Then $\mu$
is a continuous measure, and support$(\mu) = \bbR.$
\end{lemma}

\begin{proof}
Now suppose that $\mu(\{x_0\}) > 0.$ Since a Radon measure can have
at most countably many points with positive measure, there is some
$y \in \bbR$ with $\mu(\{y\}) = 0.$ But then a translation operator
maps the characteristic function of the point $x$ to that of $y$,
contradicting (H1).

If $support(\mu) \neq \bbR.$ then by definition there is a nonempty
open set $ \sO \subset \bbR$ such that $\mu(\sO) = 0.$ Let $I = (a,
b)$ be an interval contained $\sO.$ Set $ d = \frac{b-a}{2},$ and
cover $\bbR$ by intervals $I + nd,\ n \in \bbZ.$ Since this is a
countable open cover of $\bbR,$ and the measure $\mu \neq 0,$ for
some $n$ we must have that $\mu(I + nd) > 0.$ But then
\[ ||T_t(\chi_I)||_p \neq 0, \text{ but } ||\chi_I||_p = 0,\] for $t
= -nd,$ contradicting (H1).
\end{proof}

The statement that support $\mu = \bbR$ was proved in \cite{Be}; we
included a proof here for completeness.

\begin{lemma} \label{l:qti}
Let $\mu$ be a q.t.i. Radon measure on $\bbR.$ Then given a bounded
Borel set E, a sequence $\{ t_n\}$ in $\bbR$ which converges to
zero, and $\ep
> 0,$ there exists $N \in \bbN$ such that for $n \geq N,$
\[ \mu( \cup_{k=n}^{\infty} (E + t_k) ) < \mu(E) + \ep .\]
\end{lemma}

\begin{proof}
Let $I$ be a compact interval in $\bbR$ such that $ E,\ E + t_n
\subset I$ for all $n \in \bbN.$  Since $\mu(E) \leq \mu(I) <
\infty,$ there is an increasing sequence $\{C_n\}$ of compact
subsets of $E$ such that $ E \setminus \cup_{n=1}^{\infty} C_n$ is a
$\mu$-null set. But then
\begin{align*}
& \cup_{k=1}^{\infty} (E+t_k) \setminus \cup_{n=1}^{\infty}
\cup_{k=1}^{\infty} (C_n + t_k) \\
&\subset \cup_{k=1}^{\infty}( (E \setminus \cup_{n=1}^{\infty} C_n)
+ t_k)
\end{align*}
Since $\mu$ is q.t.i., $(E \setminus \cup_{n=1}^{\infty} C_n) + t_k$
is a $\mu$ null, as is the countable union of these sets.

Thus (\cite{Ru2}, Theorem 1.19) there is a compact set $C = C_n$
such that
\[ \mu(\cup_{k=1}^{\infty} (E+t_k) \setminus
\cup_{k=1}^{\infty} (C + t_k)) < \ep .\]

Let $F = \cap_{n=1}^{\infty} \cup_{k=n}^{\infty} (C + t_k).$ Thus if
$x \in F,$ there is a subsequence $t_{k_j}$ and $x_j \in C$ such
that $ x = x_j + t_{k_j}$ for all $j$.  Since $\{t_k\}$ converges to
zero, that implies $x$ is a limit point of $\{x_j\}, $ hence in $C$.
Now there is $N \in \bbN$ such that for $n \geq N$
\[ \mu(\cup_{k=n}^{\infty} (C + t_k) \setminus F) < \ep .\]
Thus, for $n \geq N,$
\begin{align*}
\mu(\cup_{k=n}^{\infty} (E+t_k) \setminus F) &\leq
\mu(\cup_{k=n}^{\infty} (E+t_k) \setminus \cup_{k=n}^{\infty}
(C+t_k))
\\
& + \mu(\cup_{k=n}^{\infty} (C+t_k) \setminus F) \\
& < \ep + \ep
\end{align*}
In other words,
\begin{align*}
\mu(\cup_{k=n}^{\infty} (E+t_k)) &< \mu(F) + 2\ep \\
    &< \mu(E) + 2\ep
\end{align*}
\end{proof}

\begin{lemma} \label{l:conttrans} Let $\mu$ be a q.t.i. Radon
measure.  Let $g$ be a bounded Borel function with bounded support,
and let $\{ t_n\}$ be a sequence which converges to $0$. Then
$\{T_{t_n}(g)\} $ converges in measure to $g$. Hence there is a
subsequence which converges $a.e. [\mu].$
\end{lemma}

\begin{proof}

Given $\ep > 0$ by Lusin's Theorem (\cite{Ru2}, Theorem 2.23) there
is a continuous function $h$ with compact support, $||h||_{\infty}
\leq ||g||_{\infty}$ and such that if
\[ E = \{x:\ |g(x) - h(x)| \geq  \ep \} \]
then $\mu(E) < \ep.$

By Lemma~\ref{l:qti}, there is $N \in \bbN$ such that for $n \geq
N,$
\[ \mu(E \cup_{k=n}^{\infty} (E+t_k)) < 2 \ep .\]
Now choose $j \geq N$ such that for all $x \in \bbR,$
\[ |h(x-t_j) - h(x)| < \ep .\]
Then for $x \notin E \cup_{k=N}^{\infty} (E + t_k),$
\begin{align*}
|g(x) - g(x - t_j)| &\leq |g(x) - h(x)| \\
    &+ |h(x) - h(t - t_j)| + |h(x-t_j) - g(x-t_j)| \\
    &< 3 \ep
\end{align*}
It follows that $\{T_{t_k}(g)\}$ converges in measure to $g$, hence
there is a subsequence converging a.e. $[\mu].$ (\cite{Ru2}, p. 73)
\end{proof}

\begin{remark} Note that if $\mu$ is a continuous Radon measure
which is not q.t.i., neither of the conclusions of
Lemmas~\ref{l:qti} or \ref{l:conttrans} need hold.

Let $C$ be the Cantor set, and $\vpi: C \to [0, 1]$ defined as
follows: if $x \in C,\ x = \sum_{n=1}^{\infty} \frac{x_n}{3^n}$
where $x_n \in \{0, 2\},$ let $\vpi(x) = \sum_{n=1}^{\infty}
\frac{x_n}{2^{n+1}}.$ Then the Cantor measure $\mu$ can be defined
by $\mu(E) = m(\vpi(E\cap C)).$ (Recall $m$ is Lebesgue measure.)

Now $\mu(C) = 1,\ \mu(C + \frac{1}{3^n}) = 0, $ so that
$T_{\frac{1}{3^n}}(\chi_C)$ does not converge to $\chi_C$ in
measure. So the conclusion of Lemma~\ref{l:conttrans} fails.

If $E = \cup_{n=1}^{\infty} (C + \frac{1}{3^n}),$ then $\mu(E) = 0,$
but $\mu(E  + t_n) = 1$ for all $n$ if $t_n = -\frac{1}{3^n}$, so
that for all $n$, $\mu( \cup_{k=n}^{\infty} (E + t_k) ) = 1.$  Thus
the conclusion of Lemma~\ref{l:qti} fails.
\end{remark}

\begin{corollary} \label{c:conttrans} Let $\mu$ satisfy (H1) and
(H2(a)).  Let $f \in L^p(\mu) $ and $\{t_k\}$ a sequence converging
to $t_0.$ Then there is a subsequence $S$ of $\{t_n\}$ such that
$T_{t_k}(f) $ converges to $T_{t_0}(f)$ a.e.$[\mu].$
\end{corollary}

\begin{proof} Let \[f_n(x) =
\begin{cases}
f(x) \text{ if } |f(x)| \leq n \text{ and } |x| \leq |n| ;\\
n \text{ if }  f(x) > n \text{ and } |x| \leq n; \\
-n \text{ if }  f(x) < -n  \text{ and } |x| \leq |n|;\\
0 \text{ if } |x| > n.
\end{cases} \]
Now each $f_n$ is equal a.e.$[\mu]$ to a bounded Borel function with
bounded support. Thus, by Lemma~\ref{l:conttrans} there is a
subsequence of $\{t_k\}$ for which $T_{t_k}(f_1) $ converges
pointwise a.e.$[\mu]$ to $T_{t_0}(f_1)$, and subsequence of that
subsequence for which $T_{t_k}(f_2)$ converges to $T_{t_0}(f_2),$
and so forth.  The standard diagonal argument gives a subsequence
$S$ for which $T_{t_k}(f_n)$ converges to $T_{t_0}(f_n)$ a.e.$[\mu]$
for all $n$, and hence $T_{t_k}(f)$  converges to $T_{t_0}(f)$ for
$t_k \in S,$ a.e.$[\mu]$.
\end{proof}

\begin{lemma} \label{l:closed}
Let $\mu$ satisfy (H1) and (H2(a)). Let $E$ be a set with
\mbox{$\mu(E) < \infty$}. For any positive $C$, the set
\[ \{ t \in \bbR:\ \mu(E+t) \leq C\mu(E)\} \]
is closed.
\end{lemma}

\begin{proof}
 Let $\{t_k\} $ be a sequence with
$\mu(E+t_k) \leq C \mu(E)$ and suppose $ t_0 = \lim t_n.$ By
Corollary~\ref{c:conttrans} there is a subsequence such that
$\chi_{E + t_k} $ converges to $\chi_{E+t_0}\ a.e.[\mu].$  Thus,
replacing the sequence by the subsequence and applying Fatou's Lemma
we have
\begin{align*}
 \mu(E+t_0) &= \int \chi_{E+t_0}(x)\, d\mu(x) \\
    &= \int \liminf \chi_{E+t_n}(x)\, d\mu(x) \\
    &\leq \liminf \int \chi_{E+t_n}(x) \, d\mu(x) \\
    &\leq C\mu(E)
 \end{align*}

 \end{proof}

\begin{lemma} \label{l:Baire} Let $\mu$ satisfy (H1) and (H2(a)).
Let $\sN = [a, b],\ a,\ b \in \bbR.$  Then there is a constant $C
> 0$ such that for all $ t \in \sN,$ and for all measureable sets
$E$ with $\mu(E) < \infty$,
\[ \mu(E+t) \leq C\mu(E) .\]
\end{lemma}

\begin{proof}For $n
= 1, 2, \dots,$ set $A_n = \{ t \in \sN: \ \mu(E+t) \leq n\mu(E) \}$
for all measureable sets $E \subset \bbR$ with $\mu(E) < \infty.$ By
Lemma~\ref{l:closed}, $A_n$ is a closed subset of $\sN.$ Since for
$n \geq ||T_t||,\ t \in A_n$ we have that
\[ \cup_{n=1}^{\infty} A_n = \sN.\]
By the Baire Category Theorem (\cite{Ru2}), some $A_N$ has interior.

Say $(a', b') \subset int A_N.$ If $ d = \frac{b'-a'}{2}$ then we
have
\[ \mu(E+t) \leq C'\mu(E) \] for all $ t \in (-d, d)$, where $C' =
||T_{-d}||C_N.$ Let $k$ be such that $\sN \subset (-kd, kd).$ Then
we can choose the constant $C$ of the Lemma to be $C = (C')^k.$
\end{proof}

\begin{corollary} \label{c:boundedtrans} Let $\mu$ satisfy (H1) and (H2(a)).
Then there is a constant $C_p$ such that
\[ ||T_t(f)||_p \leq C_p||f||_p\]
for all $ t \in \sN,\ f \in L^p(\mu).$
\end{corollary}

\begin{proof} It is enough to prove that the conclusion holds for
all $f$ in a dense subset of $L^p(\mu).$  Let $C$ be as in
Lemma~\ref{l:Baire} and let
\[ f = \sum_n a_n \chi_{E_n}\] be a simple function; in particular,
the sum is finite and the $E_n$ are disjoint, $\mu(E_n) < \infty.$
Then
\begin{align*}
||T_t(f)||_p^p &= \sum_n |a_n|^p \mu(E_n+t) \\
    &\leq \sum_n |a_n|^p C \mu(E_n) \\
    &\leq C||f||_p^p
\end{align*}
\end{proof}

Now that we know that the translation operators are uniformly
bounded in a compact neighborhood of $0$, the proof of strong
continuity of translation on $L^p(\mu)$ mimics that on $L^p(m).$

\begin{theorem} \label{t:transcont} Let $\mu$ be a Radon measure satisfying
(H1) and (H2(a)). Then $\{T_t\}$ is a $C_0$-group acting on
$L^p(\mu)$. That is
\[ \lim_{t \to 0} ||T_t(f) - f||_p = 0 \]
for all $ f \in L^p(\mu).$
\end{theorem}

\begin{proof}
Let $C$ and $\sN$ be as in Corollary~\ref{c:boundedtrans}, and let
$f \in L^p(\mu)$ and $\ep > 0$ be given.  Let $g$ be a continuous
function with compact support such that $ ||f - g||_p < \de$ where $
\de$ satisfies  $(C+1)\de < \ep/2.$. Since $T_t(g)$ converges to $g$
uniformly as $t \to 0,$ it also converges in $L^p(\mu).$ Let $
(-\eta, \eta) \subset \sN,$ where $\eta$ is sufficiently small so
that $||T_t(g) - g||_p < \ep/2$.  Then for $|t| < \eta,$
\begin{align*}
||T_t(f) - f||_p &\leq ||T_t(f-g)||_p + ||T_t(g) - g||_p + ||f-g||_p \\
    &\leq C\de + \ep/2 + \de \\
    &< \ep
\end{align*}
\end{proof}

Next we give a negative result for $L^{\infty}(\mu).$ Obviously
translation is not norm continuous on $L^{\infty}(\mu),$ but it is
not even weakly continuous.

\begin{theorem} \label{t:notwklycont} Let $\mu$ be a Radon measure satisfying
(H1) and (H2(b)). Then there exists a bounded function $f \in \sV =
L^{\infty}(\mu)$ and a bounded linear functional $\rho: \sV \to \bbR
$ such that
\[ \lim_{t \to 0} \rho(T_t(f)) \text{ exists but } \neq \rho(f) .\]
In particular, $T_t$ is not weakly continuous on $\sV.$
\end{theorem}

\begin{proof}
\textit{Claim } By Lemma~\ref{l:support}, $\mu$ is a continuous
measure. There exists a point $a \in \bbR$ such that for all $\de >
0,\ \mu([a, a+\de)) > 0.$ Suppose to the contrary. Let $I = [c, d]$
be a closed interval with $\mu(I)
> 0.$ Then for each $x \in \bbR $ there is a $\de_x > 0$ such that
$\mu([x, x+\de_x)] = 0.$  Since $I$ is compact, $I$ can be covered
by finitely many intervals of the form $(x, x+\de_x),\ x \in \bbR,\
$ say $ (x_i, x_i+\de_i),\ 1 = 1, \dots, n.$  But then
\[ 0 < \mu(I) \leq \sum_{i=1}^n \mu(x_i, x_i+\de_i) = 0.\]
This contradiction proves the claim.

Let $a \in \bbR$ be such that $\mu[a, a+\de) > 0 \ \forall \de
> 0,$ and let $ \de_n$ be a sequence decreasing to $0.$ Define a
bounded linear functional $\rho_n$ by
\[ \rho_n(f) = \frac{1}{\mu[a, a+\de_n)}\int_{[a, a+\de_n)} f\,d\mu
.\] Note that $\rho_n $ has norm $1$ on $L^{\infty}(\bbR, \mu).$ By
the Alaoglu Theorem (\cite{Ru1}, 3.15) the set $\{ \rho_n\}$ has a
limit point, say $\rho.$

Let $h = \chi_{[a, \infty)}.$ Then $\rho_n(h) = 1$ for all $n$ so
that $\rho(h) = 1.$  On the other hand, for $ \de_n < t,$
\[ \rho_n(T_t(h)) = 0 \text{ hence } \rho(T_t(h)) = 0.\]
Thus, $\rho(h) \neq \lim_{t\to 0} \rho(T_t(h)),$ completing the
proof.
\end{proof}

\section{The  measure $\mu$}
In this section we characterize those measures $\mu$ for which
$L^p(\mu)$ admits translation operators.

\begin{definition} \label{d:wb} Let $\sB$ be the space of bounded, Borel
functions on $\bbR$ with bounded support. Let $\mu$ be a Radon
measure on $\bbR.$ We say that $T_t \to T_0 = I$ as $ t \to 0$ in
the $wb$-topology if
\[ (T_t(f), g)_{\mu} := \int f(x-t) g(x)\, d\mu(x) \to (f, g)_{\mu}
:= \int f(x) g(x) \,d\mu(x) \] as $ t \to 0$, for all $f,\ g \in
\sB.$
\end{definition}

\begin{remark} If $\{T_t\}$ is a $C_0$-group acting on $L^p(\mu)$
for some $ 1 \leq p < \infty,$ then clearly $T_t \to T_0$ as $t \to
0$ in the wb-topology.
\end{remark}

\begin{lemma} \label{l:wb} With notation as above, $T_t \to T_0 = I$ as $ t \to 0$ in the
$wb$ topology on $\sB$ if and only if $\mu$ is absolutely continuous
with respect to Lebesgue measure $m$.
\end{lemma}

\begin{proof} Assume that $(T_t(f), g)_{\mu} \to (f, g)_{\mu}$ as $
t \to 0$ for all $ f,\ g \in \sB.$ We need to show that $\mu $ is
absolutely continuous with respect to Lebesgue measure.

Fix $ 0 < n < \infty,$ and let $\mu_{t, n}(A) = \mu(A\cap [-n, n] -
t)$ for $ A \in \fB,$ the $\si$-algebra of Borel sets.  Take $f =
\chi_{A\cap [-n, n]},\ g = \chi_{[-n-1, n+1]}.$ Note that $ f,\ g
\in \sB,$ and that for all $ A \in \fB,$
\begin{equation} \label{eq:2}
 \mu_{t, n}(A) = \int \chi_{A\cap [-n, n]}(x+t) \chi_{[-n-1,
n+1]}(x)\,d\mu(x) \to \mu_{0, n}(A)
\end{equation}
 as $ t \to 0.$

For $\ep > 0,$ set
\[ \la_{\ep}(A) = \frac{1}{\ep}\int_0^{\ep} \mu_{t, n}(A)\,dt,\ A
\in \fB.\]  By (\ref{eq:2}),
\begin{equation} \label{eq:3}
 \la_{\ep}(A) \to \mu_{0, n}(A) \text{ as } \ep \to 0
 \end{equation}
By Fubini's Theorem,
\begin{align*}
\la_{\ep}(A) &= \frac{1}{\ep}\int_0^{\ep}(\int_{\bbR} \chi_{A\cap
[-n, n] -t}(x)\, d\mu(x))\,dt \\
    &= \int_{\bbR}(\frac{1}{\ep}\int_0^{\ep} \chi_{A\cap [-n, n]
    -t}(x)\,dt) \,d\mu(x)
\end{align*}
If $m(A) = 0$ then for all $\ep > 0, \ x \in \bbR,\ \int_0^{\ep}
\chi_{A\cap [-n, n] -t}(x)\,dt = 0$ since $ m(A\cap [-n, n] -t) =
m(A\cap [-n, n]) \leq m(A) = 0$ using translation invariance of $m$.
Thus, $\la_{\ep}(A) = 0 \ \forall \ep > 0.$

By (\ref{eq:3}), $ \mu_{0, n}(A) = 0,$ for all $n.$  Since $\mu(A) =
\lim_n \mu_{0, n}(A),$ it follows $\mu(A) = 0.$  Thus $ \mu << m.$

Now suppose that $\mu$ is absolutely continuous with respect to $m$.
We need to show that $T_t$ is continuous at $0$ in the
$wb$-topology. Fix $f, \ g \in \sB$ and let $ h = \frac{d\mu}{dm}$
be the Radon-Nikodym derivative.

Let $M $ be such that the supports of $f,\ T_t(f),\ g$ are all
contained in $[-M, M]$ for $ |t| < 1.$  Let $\ep > 0$ be given.
Clearly,
\[ \int_{|x| \leq M} h(x)\, dm < \infty,\text{ implies } \lim_{N\to \infty} \int_{
|x|\leq M, \ h(x) > N} h(x) \, dm = 0.\] Thus, there exists $N $
such that
\[ \int_{h(x) > N,\ |x| \leq M} h(x) < \ep.\]

Since $T_t(f) \to f$ as $t \to 0$ in $L^1(m)$, there is a $\de > 0$
such that for $|t| < \de,$
\[ \int |f(x-t) - f(x)|\,dm(x) < \ep/N .\]

So for $|t| < \de,$
\begin{multline*}
|(T_t(f), g)_{\mu} - (f, g)_{\mu}| \leq \int_{ h(x) > N,\ |x| \leq
M} |f(x-t) - f(x)| |g(x)| h(x) \,dm(x) \\ + N||g||_{\infty}\int
|f(x-t) - f(x)|\,dm(x)
\\
    \leq 2||f||_{\infty} ||g||_{\infty} \int_{h(x) > N,\ |x| \leq
    M} h(x)\,dm(x) + N||g||_{\infty} \int |f(x-t) - f(x)|\,dm(x) \\
   \leq 2||f||_{\infty} ||g||_{\infty} \ep +  ||g||_{\infty} \ep
\end{multline*}
This shows that $T_t$ is $wb$-continuous at $ t = 0$, completing the
proof.

\end{proof}

\begin{lemma} \label{l:intervals} Let $\mu$ be a continuous Radon measure on the
$\si$-algebra of Borel sets. Let $E$ be a Borel set with $\mu(E) <
\infty$ and $ 0 < m(E) < \infty.$  Let $\sO$ be an open set with $ E
\subset \sO$ and $m(\sO)$ finite. Write $\sO = \cup_i (a_i, b_i)$,
where the $(a_i, b_i)$ are disjoint intervals. Then there exits $i$
such that
\[\mu((a_i, b_i)) \geq \frac{\mu(\sO)}{m(\sO)}\,(b_i - a_i) ,\] and
there exists $i$ such that
\[ \mu((a_i, b_i)) \leq \frac{\mu(\sO)}{m(\sO)}\, (b_i - a_i) .\]
In particular, if $\mu(E) = 0,$ then given $\ep > 0,\ \sO$ can be
chosen so that
\[ \mu((a_i, b_i)) < \ep\, (b_i - a_i) \]
for some $i$.
\end{lemma}

\begin{proof} Let $I_i = (a_i, b_i).$ Suppose the first
statement fails; then
 \[\text{ for all } i,\ \mu(I_i)  < \frac{\mu(\sO)}{m(\sO)}\,m(I_i).\] Hence,
\begin{align*}
\mu(\sO) &= \sum_i \mu(I_i) \\
    &< \sum_i \frac{\mu(\sO)}{m(\sO)}\,m(I_i) \\
    &< \mu(\sO)\frac{m(\sO)}{m(\sO)} \\
    &< \mu(\sO)
\end{align*}
a contradiction.

The proof of the second assertion is similar, and is omitted. The
final statement follows from the regularity of the measure $\mu.$
\end{proof}

Recall two measures $\mu$ and $\nu$ are \emph{equivalent} if $\mu$
is absolutely continuous with respect to $\nu,$ and $\nu$ is
absolutely continuous with respect to $\mu.$

\begin{lemma} \label{l:zeromeas} Suppose that $\mu$ is a nonzero Radon measure which is q.t.i.
Then $\mu$ is equivalent to Lebesgue measure.
\end{lemma}

\begin{proof}
Let $\{ t_n\}$ be a sequence in $\bbR$ which converges to $0$, and
$f,\ g$ bounded Borel functions with bounded support. Since
$\{T_{t_n}(f)\}$ converges in measure to $f$ by
Lemma~\ref{l:conttrans}, then in the notation of
Definition~\ref{d:wb},
\[ (T_{t_n}(f), g)_{\mu} \to (f, g)_{\mu} \]
so that $T_t \to I$ in the wb-topology as $ t \to 0.$ Thus by
Lemma~\ref{l:wb}, $\mu << m.$

Let $h = \frac{d\mu}{dm}$, and set $ F = \{ x:\ h(x) = 0 \}.$  Since
$\mu $ is q.t.i., $F$ is translation invariant. Define a Radon
measure $\nu$ by
\[ \nu(E) = m(E\cap F) .\]
Then
\begin{align*}
\nu(E + t) &= m((E + t)\cap F) \\
    &= m((E+t)\cap (F+t)) \\
    &= m(E \cap F + t) \\
    &= m(E \cap F) \\
    &= \nu(E)
\end{align*}

As Lebesgue measure is the unique translation-invariant measure, up
to scalar multiple, if $\nu$ is nonzero there is a positive constant
$c$ such that $ \nu(E) = c\,m(E)$ for any Borel set $E$. In
particular,
\[ \nu(F^c) = m(F\cap F^c) = 0\]
and hence $m(F^c) = 0.$  This implies that $\mu$ is the zero
measure, contrary to assumption.  Thus it must be that $\nu$ is the
zero measure, so that $m(F) = 0.$  But then $ m << \mu .$
\end{proof}

\begin{theorem} \label{t:linfty} Let $\mu$ be a Radon measure on $\bbR.$
Then the following conditions are equivalent:
\begin{enumerate}
\item $\mu$ is q.t.i.;
\item $\{T_t\}_{t \i \bbR}$ acts as a group of linear operators on $L^{\infty}(\mu);$
\item $\{T_t\}$ acts as a group of isometries on $L^{\infty}(\mu);$
\item $\mu$ is equivalent to Lebesgue measure.
\end{enumerate}
\end{theorem}

\begin{proof} Suppose $\mu$ is q.t.i. and let  $f \in L^{\infty}(\mu),\ f = \sum a_k E_k$ be
a finite sum, with the sets $E_k$ pairwise disjoint and of positive
or infinite $\mu$-measure.  By hypothesis, the sets $E+t$ have
positive or infinite $\mu$-measure, for all $t \in \bbR.$  Thus,
$T_t(f) \in L^{\infty}(\mu)$ and $||T_t(f)||_{\infty} =
||f||_{\infty}.$ Thus, $T_t$ is an isometry on a dense subspace of
$L^{\infty}(\mu),$ hence is an isometry on $L^{\infty}(\mu).$
Furthermore the $\{T_t\}$ act as a group (cf Remark~\ref{r:group}).
This shows that (1) implies (3), and clearly (3) implies (2). Also
by Remark~\ref{r:group}, q.t.i. is equivalent to (2).

Lemma~\ref{l:zeromeas} proves the equivalence of (1) and (4).
\end{proof}

\begin{remark} Recall hypothesis (H1) for $\mu$ is equivalent to
$\mu$ being q.t.i. (Lemma~\ref{l:equivclass}).  The theorem now
shows that (H1) is equivalent to (H2(b)).
\end{remark}

\begin{notation} If $h$ is a Lebesgue measureable function, and $I$
an interval in $\bbR,$ write $||h||_{I, \infty}$ to denote the
essential supremum of $h$ on $I$.
\end{notation}

\begin{theorem} \label{t:transmeas}  Let $\mu$ be a Radon measure on
$\bbR.$ Then $\mu$ satisfies (H1) and (H2(a)) if and only if $\mu$
is equivalent to Lebesgue measure, and the Radon-Nikodym derivative
$\frac{d\mu}{dm} := h$ satisfies: For each real number $L
> 0$ there is a constant $C_L$ such that if $I$ is any interval of
length $L$, then
\[ ||h||_{I,\infty}\,||1/h||_{I,\infty} \leq C_L .\]
\end{theorem}

\begin{proof} Suppose $T_t$ maps $L^p(\mu)$ to itself, for all $t \in \bbR.$
Let $L > 0$ be given, and let $I = [a, b]$ be an interval of length
$L$.

\textit{A fortiori} there is no reason why $||h||_{I, \infty}$ could
not be infinite, or similarly why  $||1/h||_{I, \infty}$ could not
be infinite. From Lemma~\ref{l:zeromeas} we know that $\mu$ is
equivalent to Lebesgue measure, and hence $||h||_{I, \infty}$ cannot
be zero, and $||1/h||_{I,\infty}$ cannot be zero. We will show in
fact that both are finite. Let $M$ and $K$ be constants such that
\[ 0 < M < ||h||_{I,\infty} \]
and
\[ 0 < K < ||1/h||_{I, \infty} .\]

Let $E = \{ x \in I:\ h(x) \geq  M\}$ and $F = \{ x \in I:\ 1/h(x)
\geq K \}.$ By definition of essential supremum, both of these sets
have positive Lebesgue measure, not exceeding $L = m(I).$ Let $\sO$
be an open cover of $E$ and $\sP$ an open cover of $F$, so that both
$\sO,\ \sP$ are contained in the interval $ [a-1, b+1].$
Furthermore, given $\ep
> 0,$ we can require that $m(\sO) < m(E) + \ep$ and $\mu(\sP) < \mu(F) + \ep.$

If $\sO = \cup_i I_i$ where the $I_i$ are pairwise disjoint, open,
then by Lemma~\ref{l:intervals} there is an interval $I' = I_i$ for
which $ \mu(I') \geq \frac{\mu(\sO)}{m(\sO)}\,m(I').$ Hence,
\begin{align*}
\mu(I') &\geq \frac{\mu(E)}{m(\sO)}m(I') \\
    &\geq \frac{M\,m(E)}{m(E)+\ep}\,m(I')
\end{align*}

Writing $\sP$ as the disjoint union of pairwise disjoint open sets,
there is, by the same Lemma, an interval $J$ in $\sP$ for which
$\mu(J) \leq \frac{\mu(\sP)}{m(\sP)}\,m(J).$ Furthermore, we may
assume the length of $J$ is less than the length of $I'$. Indeed, if
this is not the case then bisect $J$ so $J = J_1 \cup J_2.$ Then
$J_i $ satisfies $J_i \subset \sP$ and $\mu(J_i) <
\frac{\mu(\sP)}{m(\sP)}\, m(J_i)$ for at least one of $i = 1, \ 2.$
Continue bisecting until an interval is obtained with length less
than $I'.$ Rename that interval $J.$

Hence,
\begin{align*}
\mu(J) &\leq \frac{\mu(F) + \ep}{m(F)}\,m(J) \\
    &\leq \frac{K^{-1}\,m(F) + \ep}{m(F)}\,m(J)
\end{align*}

Cover $I'$ by finitely many translates of $J$,
\[ I' \subset \cup_{k=1}^n (J+ t_k) \]
where $m(\cup_k (J+t_k)) = n\,m(J) \leq 2 m(I').$  Also, since the
intervals $I',\ J$ are subsets of an interval of length $L+2,$ we
have that $ |t_k| \leq L+2,\ 1 \leq k \leq n.$

Let $C$ be the constant in Corollary~\ref{c:boundedtrans}
corresponding to the neighborhood $[L-2, L+2].$ Then
\[ \mu(J+t_k) \leq C\mu(J),\ 1 \leq k \leq n.\]
Hence,
\begin{align*}
\frac{m(E)}{m(E)+\ep}\,M m(I') &\leq \mu(I') \\
    &\leq \mu(\cup_{k=1}^n J + t_k ) \\
    &\leq n\,C\,\mu(J) \\
    &\leq C\,\frac{K^{-1}\,m(F)+\ep}{m(F)}\,
    n\,m(J)\\
    &\leq
    C\,\frac{K^{-1}\,m(F)+\ep}{m(F)}\,2\,m(I')
\end{align*}
Cancelling $m(I')$ from both sides of the inequality, we obtain
\[ \frac{m(E)}{m(E)+\ep}\,M\,\frac{m(F)}{K^{-1}\,m(F) +\ep} \leq 2C .\]
Since this is true for all $\ep > 0$ we conclude
\[ M\, K \leq 2C .\]
However, $M$ was arbitrary in the interval $ (0, ||h||_{I,\infty})$
and $K$ was arbitrary in the interval $(0, ||1/h||_{I,\infty})$ so
it follows that
\[ ||h||_{I, \infty}\, ||1/h||_{I, \infty} \leq 2C .\]
Thus, we can take the constant $C_L = 2C.$

For the converse, suppose the Radon-Nikodym derivative $h =
\frac{d\mu}{dm}$ satisfies the condition of the Theorem. It is
enough to show that each translation operator $T_t$ is a bounded
operator on $L^p(\mu),\ 1 \leq p < \infty.$ So fix $t \in \bbR$ and
let $ L = 2|t|.$ So if $E$ is any measurable subset of $\bbR$
contained in an interval of length $L/2$ then both $E,\ E+t$ are
contained in an interval of length $L$. Hence, $\mu(E+t) \leq
C_L\,\mu(E).$ Let $f$ be a simple function with compact support in
$L^p(\mu)$ and write
\[ f = \sum_k a_k\, E_k \]
where each $E_k$ is contained in an interval of length $L/2.$ Then
$f \in L^p(\mu)$. Also note that the set of such functions is dense
in $L^p(\mu)$ for $ 1 \leq p < \infty.$ Then
\begin{align*}
{||T_t(f)||_p}^p &= \sum_k |a_k|^p \mu(E_k + t) \\
    &\leq \sum_k |a_k|^p C_L \mu(E_k) \\
    &\leq C_L {||f||_p}^p
\end{align*}
Since $T_t$ is bounded by ${C_L}^{\frac{1}{p}}$ on a dense subset,
it is bounded on $L^p(\mu).$
\end{proof}

\begin{corollary} \label{c:allp} If $\mu$ is a Radon measure on
$\bbR$ and if the group $\{T_t\}$ maps $L^p(\mu)$ into itself for
some $p, \ 1 \leq p < \infty,$  then $\{T_t\}$ maps $L^p(\mu)$ into
itself for all $p, \ 1 \leq p \leq \infty.$
\end{corollary}

\begin{example} \label{e:transmu} Let $d\mu = h\,dm$ where $h(x) =
\exp(c|x|^{\al}).$ Then $\mu$ satisfies the conditions of the
Theorem for any $c \in \bbR$ and $ 0 \leq \al \leq 1.$
\end{example}

%\pagebreak

\section{Relation to the work of D. Bell}
Given a Radon measure $\mu$ on $(\bbR, \fB(\bbR))$, define $\mu_t$
to be the translated measure, $\mu_t(A) = \mu(A - t).$ Note that
$\mu$ is quasi-translation invariant (Definition~\ref{d:qti}) if and
only if $\mu$ and $\mu_t$ are equivalent measures, for all $t \in
\bbR.$

\begin{theorem} \label{t:bell} (Bell~\cite{Be}) The following are equivalent:
\begin{enumerate}
\item  $\mu$ is q.t.i. on $\bbR$, and the family of Radon-Nikodym
derivatives
\[ h(t, x) = \frac{d\mu_t}{d\mu} \]
is separately continuous in $t$ and $x$.
\item $\mu$ is equivalent to Lebesgue measure $m$ with density $f =
\frac{d\mu}{dm}$, and $f$ is continuous and positive on $\bbR.$
\end{enumerate}
\end{theorem}

\begin{proof} Suppose $(1)$ holds. By Theorem~\ref{t:linfty}, $\mu $
is equivalent to $m$. If $f = \frac{d\mu}{dm}, $ then $f(x-t) =
\frac{d\mu_t(x)}{dm(x)}.$ Hence,
\[f(x-t) = \frac{d\mu_t(x)}{dm(x)} =
\frac{d\mu_t(x)}{d\mu(x)}\,\frac{d\mu(x)}{dm(x)} = h(t, x)f(x) .\]
Now for each $x \in \bbR$ there exits $t \in \bbR$  such that
$f(x-t) > 0 $. But then the formula implies $f(x) > 0.$  Thus $f$ is
strictly positive. Since $h$ is continuous in  $t$ for fixed $x$,
that implies that $f(t) = h(-t, 0)f(0)$ is continuous in $t$.

Now assume that $(2)$ holds. By Theorem~\ref{t:linfty} $\mu$ is
q.t.i. and the formula in the first part of the proof implies that
$h$ is separately (in fact, jointly) continuous in $t$ and $x$.
\end{proof}

\begin{remark} \label{r:comparison} Theorem~\ref{t:linfty} is not the
same as Theorem~\ref{t:bell}. Theorem~\ref{t:linfty} says that $\mu$
is q.t.i. if and only if $\mu$ is equivalent to $m$, whereas
Theorem~\ref{t:bell} assumes more about $\mu$ and asserts more.
\end{remark}

\end{document}